\documentclass[12pt]{article}
\UseRawInputEncoding

{\begin{Sbox}\begin{minipage}}%
{\end{minipage}\end{Sbox}\fbox{\TheSbox}}%

\usepackage[psamsfonts]{amssymb}
\usepackage{amsmath, amsthm, anysize, enumerate, color, url}
\usepackage{graphicx}
\usepackage{epstopdf}
\usepackage{epsfig}
\usepackage{subfigure}
\usepackage{algorithm}
\usepackage{algpseudocode}
\usepackage{natbib}
\usepackage{booktabs}
\usepackage{threeparttable}
\usepackage[colorlinks,citecolor=blue,urlcolor=blue]{hyperref}
\usepackage{breakurl}

\newtheorem{theorem}{Theorem} 

\newtheorem{corollary}[theorem]{Corollary}

\newtheorem{lemma}{Lemma} 

\newtheorem{assumption}{Assumption}
\theoremstyle{definition}
\newtheorem{remark}{Remark}  

\newcommand{\R}{\mathbb{R}}

\newcommand{\bs}{\boldsymbol}
\newcommand{\Z}{\mathbb{Z}}

\renewcommand{\baselinestretch}{1.2}
\usepackage{setspace}

\begin{document}
\title{A sparse $p_0$  model with covariates for directed networks}
\author{
Qiuping Wang\thanks{Department of Statistics, Central China Normal University, Wuhan, 430079, China.
\texttt{Emails:}qp.wang@mails.ccnu.edu.cn.} 
\\
Central China Normal University\\
}
\date{}

\maketitle
\begin{abstract}

We are concerned here with unrestricted maximum likelihood estimation in a sparse $p_0$ model with covariates for directed networks.
The model has a density parameter $\nu$, a $2n$-dimensional node parameter $\bs{\eta}$ and a fixed dimensional regression coefficient $\bs{\gamma}$ of covariates. Previous studies focus on the restricted likelihood inference.
When the number of nodes $n$ goes to infinity, we derive the $\ell_\infty$-error between the maximum likelihood estimator (MLE)
$(\widehat{\bs{\eta}}, \widehat{\bs{\gamma}})$ and its true value $(\bs{\eta}, \bs{\gamma})$.
They are $O_p( (\log n/n)^{1/2} )$ for $\widehat{\bs{\eta}}$ and $O_p( \log n/n)$ for $\widehat{\bs{\gamma}}$, up to
an additional factor. This explains the asymptotic bias phenomenon in the asymptotic normality of
$\widehat{\bs{\gamma}}$ in \cite{Yan-Jiang-Fienberg-Leng2018}.
Further, we derive the asymptotic normality of the MLE.
Numerical studies and a data analysis demonstrate our theoretical findings.
\vskip 5 pt \noindent
\textbf{Key words}: Asymptotic normality; Consistency; Degree heterogeneity;  Homophily; Maximum likelihood estimator. \\

\end{abstract}
\vskip 5pt

\section{Introduction}

The degree heterogeneity is a very common phenomenon in realistic networks, which characterizes the variation in the node degrees.
In directed networks, there are in- and out-degrees for nodes.
Another distinct phenomenon is homophily that links tend to form between nodes with same attributes such as age and sex.
As depicted in  \cite{Yan-Jiang-Fienberg-Leng2018}, the directed friendship network between $71$ lawyers studied in \cite{Lazega:2001}
displays a strong homophily effect. \cite{Yan-Jiang-Fienberg-Leng2018} incorporated the covariates into the $p_0$ model in \cite{yan2016aos}
to address the homophily. However, the network sparsity has not been considered in their paper.
This is an equally important feature in realistic networks.

Here, we further incorporate a sparsity parameter to \citeauthor{Yan-Jiang-Fienberg-Leng2018}'s model.
Consider a directed graph $\mathcal{G}_{n}$ on $n \geq 2$ nodes that are labeled by $1, \ldots, n$.
Let $a_{ij}\in\{0, 1\}$ be an indictor whether there is a directed edge
from node $i$ pointing to $j$. That is, if there is a directed edge from $i$ to $j$,
then $a_{ij}=1$; otherwise, $a_{ij}=0$.
Our model postulates that $a_{ij}$'s follow independent Bernoulli distributions such that a directed link exists from node $i$ to node $j$ with probability
\begin{equation}\label{eq-model}
P( a_{ij}= 1) =
\frac{ \exp(  \nu+ \alpha_{i}+ \beta_{j} + Z_{ij}^\top \bs{\gamma} ) }{ 1 + \exp( \nu+ \alpha_{i}+ \beta_{j} + Z_{ij}^\top \bs{\gamma} ) }.
\end{equation}
In this model, the parameter $\nu$ measures the network sparsity.
For a node $i$, the incomingness parameter denoted by $\beta_i$ characterizes how attractive the node
and an outgoingness parameter denoted by $\alpha_{i}$ illustrates the extent to which the node is attracted to others [\cite{Holland:Leinhardt:1981}].
$\bs{\gamma}$ is the regression coefficient of covariates $Z_{ij}$'s.
A larger $Z_{ij}^\top \bs{\gamma}$ implies a higher likelihood for node $i$ and $j$ to be connected.
We will call the above model ``covariate-$p_0$-model" hereafter.

The covariate $Z_{ij}$ is either a link dependent vector or a function of node-specific covariates.
If $X_i$ denotes  a vector of node-level attributes, then these node-level attributes can be used to construct a $p$-dimensional vector $Z_{ij}=g(X_i, X_j)$, where
$g(\cdot, \cdot)$ is a function of its arguments. For instance, if we let $g(X_i, X_j)$ equal to $\|X_i - X_j\|_1$, then it measures the similarity between node $i$ and $j$ features.

Exploring asymptotic theory in the covaraite-$p_0$-model is challenging [e.g. \cite{Fienberg:2012,Yan-Jiang-Fienberg-Leng2018}].
The restricted maximum likelihood inference has been used to derive the consistency and asymptotic normality of the restricted
maximum likelihood estimator in \cite{Yan-Jiang-Fienberg-Leng2018}.
However, it is often to solve unrestricted maximum likelihood problem in practices since
it is difficult to determine the range of unknown parameters, especially in a situation with a growing dimension of parameter space.
Therefore, developing unrestricted likelihood inference is more interesting and more actually related.

Motivated by techniques for the analysis of the unrestricted likelihood inference in the sparse $\beta$-model with covariates for undirected graphs
in \cite{Yan:2020a},
we generalize his approach to directed graphs here.
When the number of nodes $n$ goes to infinity, we derive the $\ell_\infty$-error between the maximum likelihood estimator (MLE)
$(\widehat{\bs{\eta}}, \widehat{\bs{\gamma}})$ and its true value $(\bs{\eta}, \bs{\gamma})$ by using a two-stage Newton process
that first finds the error bound between $\widehat{\bs{\eta}}_\gamma$ and $\bs{\eta}$
with a fixed $\bs{\gamma}$ and then derives the error bound between
$\widehat{\bs{\gamma}}$ and $\bs{\gamma}$.
They are $O_p( (\log n/n)^{1/2} )$ for $\widehat{\bs{\eta}}$ and $O_p( \log n/n)$ for $\widehat{\bs{\gamma}}$, up to
an additional factor on parameters.
Further, we derive the asymptotic normality of the MLE. The asymptotic distribution of $\widehat{\bs{\gamma}}$
has a bias term while $\widehat{\bs{\eta}}$ does not have such a bias,
which collaborates the findings in \cite{Yan-Jiang-Fienberg-Leng2018}. This is because of
different convergence rates for $\widehat{\bs{\gamma}}$ and $\widehat{\bs{\eta}}$, which
explains the bias phenomenon in  \cite{Yan-Jiang-Fienberg-Leng2018}.
Wide simulations are carried out to demonstrate our theoretical findings.
A real data example is also provided for illustration.

\subsection{Related work}
There is a large body of work concerned with the degree heterogeneity.
The $\beta$-model is one of the most popular models to model  the degree heterogeneity,
which is named by \cite{Chatterjee:Diaconis:Sly:2011}.
Since the number of parameters increases with the number of nodes, asymptotic theory is nonstandard.
Exploring the properties of the $\beta$-model and its generalizations has received wide attentions in recent years
\citep{Chatterjee:Diaconis:Sly:2011, Perry:Wolfe:2012, Hillar:Wibisono:2013, Yan:Xu:2013, Rinaldo2013}.
In particular, \cite{Chatterjee:Diaconis:Sly:2011} proved the uniform consistency of the maximum likelihood estimator (MLE) and \cite{Yan:Xu:2013} derived the asymptotic normality of the MLE.
\cite{Yan:Qin:Wang:2015} further established a unified theoretical framework under which
the consistency and asymptotic normality of the moment estimator hold in a class of node-parameter network models.
In the directed case, \cite{yan2016aos} proved the consistency and asymptotic normality of the MLE in the $p_0$ model,
which is a  special case of the $p_1$ model by \cite{Holland:Leinhardt:1981}.

\cite{Graham:2017} incorporated covariates to the $\beta$-model to model the homophily in undirected networks.
\cite{Dzemski:2017} and \cite{Yan-Jiang-Fienberg-Leng2018} proposed directed graph models to model edges with covariates by using a
bivariate normal distribution and a logistic distribution, respectively.
The asymptotic normality of the restricted maximum likelihood
estimator of the covariate parameter were derived under the assumption  that the estimators for
all parameters are taken in one compact set in these papers.
\cite{Yan-Jiang-Fienberg-Leng2018} further derived the asymptotic normality of the restricted MLE of the node parameter.
However, the convergence rate of the covariate parameter is not investigated in their works.
\cite{Yan:2020a} worked with the unrestricted maximum likelihood estimation in a sparse $\beta$-model with covariates for
undirected graphs
and proved the uniformly consistency and asymptotic normality of the MLE,
where the consistency is proved via the convergence rate of the Newton iterative sequence in a two-stage Newton method.
We demonstrate that his method could be further extended here.
We also note that extending his result to directed graphs is not trivial because of that the Fisher information matrix here is very different from
that in the sparse $\beta$-model with covariates and that the dimension of parameter space here is much bigger.

Finally, we note that the model can be recast into a logistic regression model.
There have been a considerable mount of literatures on asymptotic theory in generalized linear models with a growing number of parameters.
\cite{haberman1977} establish the asymptotic theory of the MLE in exponential response models in a setting that
allows the number of parameters goes to infinity.
By assuming that a sequence of independent and identical distributed samples $\{ X_i\}_{i=1}^n$ from
a canonical exponential family $c(\theta) \exp(\theta^\top x)$ with an increasing dimension $p_n$,
\cite{portnoy1988} shew $\|\widehat{\theta} - \theta^*\|_\infty =O_p( \sqrt{p_n/n})$ if $p_n/n=o(1)$
and $\widehat{\theta}$ is asymptotically normal if $p_n^2/n=o(1)$ under some additional conditions.
For clustered binary data $\{(Y_i, X_i)\}_{i=1}^n$ with a $p_n$ dimensional covariate $X_i$ and regression parameter $\theta$, \cite{wang2011} derived the asymptotic theory of generalized estimating equations estimators
 if $p_n^2/n=o(1)$.
\cite{liang2012} establish the asymptotic
normality of the maximum likelihood estimate when the number of covariates $p$ goes to infinity with the sample size $n$ in the order of $p = o(n)$,
but the error bound between $\widehat{\theta}$ and $\theta$ is not investigated. We note that asymptotic theory in these papers needs
  the condition that $c_1 N \le \lambda_{\min} \le \lambda_{\max}  \le c_2 N$ for two positive constants $c_1$ and $c_2$,
where $S_N=\sum_{i=1}^N x_i x_i^\top$ and $\lambda_{\min}$ and $\lambda_{\max}$ are the minimum and maximum eigenvalue of $S_N$,
 $x_i$ is a $p_N$-dimensional covariate vector.
 In our model, 
the first $n$ diagonal entries of $S_N$ is in the order of $n$
while the last $p$ diagonal entries of $S_N$ is in the order of $n^2$.
Because of different orders of the diagonal elements of $S_N$, the ratio $\lambda_{\max}/\lambda_{\min}$ is in the order of $O(n^2)$,
 far away from the constant order in \cite{liang2012}.
It is also worth noting that the consistency and asymptotic normality of the MLE have been derived in exponential
response model [\cite{portnoy1988}].
The asymptotic normality of $\widehat{\theta}$ does not contain a bias term in all mentioned these papers
while the asymptotic distribution of $\widehat{\gamma}$ in our model contains a bias term.
This is referred to as the incident bias problem in economic literature [e.g., \cite{Graham:2017}].

For the remainder of the paper, we proceed as follows.
In Section \ref{section:model}, we give the details on the model considered in this paper.
In section \ref{section:main}, we establish asymptotic results.
Numerical studies are presented in Section \ref{section:simulation}.  We provide further discussion and future work in Section \ref{section:summary}.
All proofs are relegated to the appendix.

\section{Maximum Likelihood Estimation}
\label{section:model}

Let $\bs{\alpha}=(\alpha_1, \ldots, \alpha_n)^\top$ and $\bs{\beta}=(\beta_1, \ldots, \beta_n)^\top$.
If one transforms $(\nu, \bs{\alpha}, \bs{\beta})$ to $(\nu+2c_2, \bs{\alpha}-c_1-c_2, \bs{\beta}+c_1-c_2)$
for two distinct constants $c_1$ and $c_2$, the probability in \eqref{eq-model} does not change.
Thus, we need restrictions on parameters to ensure the model identification.
We set the restriction that $\alpha_n=0$ and $\beta_n=0$ for practical analysis, which keeps the density parameter $\nu$.
For theoretical studies in next section, we set $\nu=0$ and $\beta_n=0$ as in \cite{Yan-Jiang-Fienberg-Leng2018} for convenience.
Other restrictions are also possible, e.g., $\sum_i \alpha_i =0$ and $\sum_j \beta_j=0$, or, $\nu=0$ and $\sum_i \alpha_i=0$.

Denote $A=(a_{ij})_{n\times n}$ as the adjacency matrix of $\mathcal{G}_{n}$.
We assume that there are no self-loops, i.e., $a_{ii}=0$.
We use $d_{i}= \sum_{j \neq i} a_{ij}$ to denote the out-degree of node $i$ and $\mathbf{d}=(d_1, \ldots, d_n)^\top$.
Similarly, $b_j = \sum_{i \neq j} a_{ij}$ denotes the in-degree of node $j$ and $\mathbf{b}=(b_1, \ldots, b_n)^\top$.
The pair $\{(b_1, d_1), \ldots, (b_n, d_n)\}$ is the so-called bi-degree sequence.
Since an out-edge from node $i$ pointing to $j$ is the in-edge of $j$ coming from $i$, it is immediate that
the sum of out-degrees is equal to that of in-degrees.

Since the random variables $a_{i,j}$ for $i \neq j$, are mutually independent given the covariates,
the log-likelihood function is
\begin{equation}\label{Eq:density:whole}
\ell( \nu, \bs{\alpha}, \bs{\beta}, \bs{\gamma})
 =   \sum_{i,j}a_{ij} (\nu+Z_{ij}^\top \bs{\gamma}) + \bs{\alpha}^\top \mathbf{d} + \bs{\beta}^\top \mathbf{b} -
 \sum_{i\neq j} \log \big( 1 + \exp( Z_{ij}^\top \bs{\gamma} + \nu+ \alpha_{i}+ \beta_{j} )\big).
\end{equation}
 As discussed before, $\bs{\alpha}$ is a parameter vector tied to the out-degree sequence, and $\bs{\beta}$ is a parameter vector tied to the in-degree sequence,
and $\bs{\gamma}=(\gamma_1, \ldots, \gamma_p)^\top$ is a parameter vector tied to the information of node covariates.
It follows that the score equations are
\renewcommand{\arraystretch}{1.3}
\begin{equation}\label{eq:likelihood-binary}
\large
\begin{array}{rcl}
\sum\limits_{i,j=1;i \neq j}^n a_{ij} &=&
\sum\limits_{i,j=1;i \neq j}^n \frac{ e^{Z_{ij}^\top \bs{{\gamma}} + \nu+\alpha_i + \beta_j } }
{1+ e^{Z_{ij}^\top \bs{\gamma} + \nu+ \alpha_i + \beta_j } }, \\
\sum\limits_{i,j=1;i \neq j}^n a_{ij}Z_{ij} &=&
\sum\limits_{i,j=1;i \neq j}^n \frac{Z_{ij}e^{Z_{ij}^\top \bs{{\gamma}} + \nu+ \alpha_i + \beta_j } }
{1+ e^{Z_{ij}^\top \bs{\gamma} + \alpha_i + \beta_j } }, \\
d_i & = & \sum\limits_{k=1, k\neq i}^n \frac{ e^{ Z_{ij}^\top \bs{\gamma} + \nu+ \alpha_i + \beta_k } }
{ 1+ e^{Z_{ij}^\top \bs{\gamma} + \nu+ \alpha_i + \beta_k}},~~~i=1,\ldots, n, \\
b_j & = & \sum\limits_{k=1,k\neq j}^n \frac{ e^{Z_{ij}^\top \bs{\gamma} + \nu + \alpha_k + \beta_j }}
{ 1 + e^{Z_{ij}^\top \bs{\gamma} + \nu + \alpha_k + \beta_j } },~~j=1,\ldots, n.
\end{array}
\end{equation}
Denote the MLE of $\ell ( \nu, \bs{\alpha}, \bs{\beta}, \bs{\gamma})$ as $(\widehat{\nu},
\widehat{\bs{\alpha}}, \widehat{\bs{\beta}}, \widehat{\bs{\gamma}})$. Note that
$\widehat{\alpha}_1=\alpha_1=0$ and $\widehat{\beta}_1=\beta_1=0$. If $(\widehat{\nu},
\widehat{\bs{\alpha}}, \widehat{\bs{\beta}}, \widehat{\bs{\gamma}})$ exists,
then it is unique and satisfies the above equations.

As discussed in \cite{Yan-Jiang-Fienberg-Leng2018},
when the number of nodes $n$ is small, we can simply use the R function ``glm" to solve \eqref{eq:likelihood-binary}. For relatively large $n$, this is no longer feasible as it is memory demanding to store the design matrix needed for $\bs{\alpha}$ and $\bs{\beta}$. In this case,
we recommend the use of a two-step iterative algorithm by alternating between solving the second and third equations in  \eqref{eq:likelihood-binary} via the fixed point method in \cite{yan2016aos} and
solving the first equation in \eqref{eq:likelihood-binary} via some existing algorithm
for generalized linear models.

\section{Theoretical Properties}
\label{section:main}

In this section, we set $\nu=0$ and $\beta_n=0$ as the model identification condition.
The asymptotic properties of $\widehat{\nu}$ in the restriction $\alpha_n=\beta_n=0$ is the same as those of
$\widehat{\alpha}_n$ in the restriction $\nu=\beta_n=0$.

We first introduce some notations. Let $\R = (-\infty, \infty)$ be the real domain. For a subset $C\subset \R^n$, let $C^0$ and $\overline{C}$ denote the interior and closure of $C$, respectively.
For a vector $\mathbf{x}=(x_1, \ldots, x_n)^\top\in \R^n$, denote by $\|\mathbf{x}\|$ for a general norm on vectors with the special cases
$\|\mathbf{x}\|_\infty = \max_{1\le i\le n} |x_i|$ and $\|\mathbf{x}\|_1=\sum_i |x_i|$ for the $\ell_\infty$- and $\ell_1$-norm of $\mathbf{x}$ respectively.
When $n$ is fixed, all norms on vectors are equivalent. Let
$B(\mathbf{x}, \epsilon)=\{\mathbf{y}: \| \mathbf{x} - \mathbf{y} \|_\infty \le \epsilon\}$ be an $\epsilon$-neighborhood of $\mathbf{x}$.
For an $n\times n$ matrix $J=(J_{i,j})$, let $\|J\|_\infty$ denote the matrix norm induced by the $\ell_\infty$-norm on vectors in $\R^n$, i.e.,
\[
\|J\|_\infty = \max_{x\neq 0} \frac{ \|J\mathbf{x}\|_\infty }{\|\mathbf{x}\|_\infty}
=\max_{1\le i\le n}\sum_{j=1}^n |J_{i,j}|,
\]
and $\|J\|$ be a general matrix norm.
Define the matrix maximum norm: $\|J\|_{\max}=\max_{i,j}|J_{ij}|$.
The notation $\sum_i$ denotes the summarization over all $i=1, \ldots, n$ and
$\sum_{i\neq j}$  is a shorthand for $\sum_{i=1}^n \sum_{j=1, j\neq i}^n$.

For convenience of our theoretical analysis, define $\bs{\eta}=(\alpha_1, \ldots, \alpha_n, \beta_1, \ldots, \beta_{n})^\top$.
We use the superscript ``*" to denote the true parameter under which the data are generated.
When there is no ambiguity, we omit the super script ``*".
Further, define
\[
\pi_{ij}:=\alpha_i+\beta_j+Z_{ij}^\top \bs{\gamma}, ~~ \mu( x ): = \frac{ e^x }{ 1 + e^x }.
\]
Write $\mu^\prime$, $\mu^{\prime\prime}$ and $\mu^{\prime\prime\prime}$ as the first, second and third derivative of $\mu(x)$ on $x$, respectively.
Direct calculations give that
\[
\mu^\prime(x) = \frac{e^x}{ (1+e^x)^2 },~~  \mu^{\prime\prime}(x) = \frac{e^x(1-e^x)}{ (1+e^x)^3 },~~ \mu^{\prime\prime\prime}(x) = \frac{e^x(1-4e^x+e^{2x})}{ (1+e^x)^4 }.
\]
It is not difficult to verify
\begin{equation}\label{ineq-mu-third-derivative}
|\mu^\prime(x)| \le \frac{1}{4}, ~~ |\mu^{\prime\prime}(x)| \le \frac{1}{4},~~ |\mu^{\prime\prime\prime}(x)| \le \frac{1}{4}.
\end{equation}

Let $\epsilon_{n1}$ ($=o(1)$) and $\epsilon_{n2}$ ($\le (\log n/n)^{1/2}/(pz_{\max})$) be two small positive numbers, where
$z_{\max}:=\sup_{i,j} \|Z_{ij}\|_\infty$. Define
\begin{equation}\label{eq-definition-bn}
b_n := \sup_{ \bs{\eta}\in B(\bs{\eta}^*, \epsilon_{n1}), \bs{\gamma} \in B(\bs{\gamma}^*, \epsilon_{n2})}
\frac{ (1+e^{\pi_{ij}})^2 }{ e^{\pi_{ij}}  } = O( e^{ 2\|\bs{\eta}^*\|_\infty + z_{\max}\|\bs{\gamma}^*\|_\infty}).
\end{equation}
It is easy to see that $b_n \ge 4$.
When $\bs{\eta} \in B(\bs{\eta}^*, \epsilon_{n1}), \bs{\gamma}\in B(\bs{\gamma}^*, \epsilon_{n2})$,
$\mu^\prime(\pi_{ij}) \ge 1/b_n$.

When causing no confusion, we will simply write $\mu_{ij}$ stead of $\mu_{ij}(\bs{\eta}, \bs{\gamma})$ for shorthand,
where
\[
\mu_{ij}(\bs{\eta}, \bs{\gamma}) =
\frac{ \exp( \alpha_i + \beta_j + Z_{ij}^\top \bs{\gamma} ) }{ 1 + \exp(\alpha_i + \beta_j + Z_{ij}^\top \bs{\gamma} ) } = \mu(\pi_{ij}).
\]
Write the partial derivative of a function vector $F(\widehat{\bs{\eta}}, \bs{\gamma})$ on $\bs{\eta}$ as
\[
\frac{ \partial F(\widehat{\bs{\eta}}, \widehat{\bs{\gamma}}) }{ \partial \bs{\eta}^\top }
=\frac{ \partial F(\bs{\eta}, \bs{\gamma}) }{ \partial \bs{\eta}^\top }\bigg |_{\bs{\eta}=\widehat{\bs{\eta}},
\bs{\gamma}=\widehat{\bs{\gamma}}}.
\]

Throughout the remainder of this paper, we make the following assumption.
\begin{assumption}
Assume that $p$, the dimension of $Z_{ij}$, is fixed and that the support of $Z_{ij}$ is $\Z^p$, where $\Z$ is a compact subset of $\R$.
\end{assumption}

The above assumption is made in \cite{Graham:2017} and \cite{Yan-Jiang-Fienberg-Leng2018}.
In many real applications, the attributes of nodes have a fixed dimension
and $Z_{ij}$ is bounded.
For example, if $Z_{ij}$'s are indictor variables such as sex, then the assumption holds.
If $Z_{ij}$ is not bounded, we make the transform $\tilde{Z}_{ij} = e^{Z_{ij}}/( 1 + e^{Z_{ij}})$.

\subsection{Consistency}\label{subsection:ar}

In order to establish the existence and consistency of $(\widehat{\bs{\eta}}, \widehat{\bs{\gamma}})$,
 we define a system of functions
\begin{eqnarray}
\nonumber
F_{i}(\bs{\eta}, \bs{\gamma}) &  =  &
\sum_{k=1; k \neq i}^n \frac{ e^{\alpha_i+\beta_k+ Z_{ij}^\top \bs{\gamma}} }{1+e^{\alpha_i+\beta_k+ Z_{ij}^\top \bs{\gamma}} } - d_i,~~~i=1,\ldots, n, \\
\label{eq:Fgamma}
F_{n+j}(\bs{\eta}, \bs{\gamma}) & = &  \sum_{k=1; k\neq j}^n \frac{e^{\alpha_k+\beta_j + Z_{kj}^\top \bs{\gamma}} }
{1+e^{\alpha_k+\beta_j+ + Z_{kj}^\top \bs{\gamma}} } - b_j,~~~j=1,\ldots, n, \\
\nonumber
F(\bs{\eta}, \bs{\gamma}) & = & (F_{1}(\bs{\eta}, \bs{\gamma}), \ldots,
F_n(\bs{\eta}, \bs{\gamma}), F_{n+1}(\bs{\eta}, \bs{\gamma}), \ldots, F_{2n-1}(\bs{\eta}, \bs{\gamma}))^\top,
\end{eqnarray}
which are based on the score equations for $\widehat{\bs{\eta}}$.
Define $F_{\gamma, i}(\bs{\eta})$ be the value of $F_{i}(\bs{\eta}, \bs{\gamma})$ when $\bs{\gamma}$ is fixed.
Let $\widehat{\bs{\eta}}_\gamma$ be a solution to $F_\gamma( \bs{\eta})=0$ if it exists.
Correspondingly, we define two functions for exploring the asymptotic behaviors of the estimator of $\bs{\gamma}$:
\begin{eqnarray}
\label{definition-Q}
Q(\bs{\eta}, \bs{\gamma})= \sum_{i,j=1;i\neq j}^n Z_{ij} (   \mu_{ij}(\bs{\eta}, \bs{\gamma}) - a_{ij}  ), \\
\label{definition-Qc}
Q_c(\bs{\gamma})= \sum_{i,j=1;i\neq j}^n Z_{ij} (    \mu_{ij}(\widehat{\bs{\eta}}_\gamma, \bs{\gamma}) - a_{ij}  ).
\end{eqnarray}
$Q_c(\bs{\gamma})$ could be viewed as the concentrated or profile function of $Q(\bs{\eta}, \bs{\gamma})$ in which the degree parameter $\bs{\eta}$ is profiled out.
It is clear that
\begin{equation*}\label{equation:FQ}
F(\widehat{\bs{\eta}}, \widehat{\bs{\gamma}})=0,~~F_\gamma( \widehat{\bs{\eta}}_\gamma)=0,~~
Q(\widehat{\bs{\eta}}, \widehat{\bs{\gamma}})=0,~~Q_c( \widehat{\bs{\gamma}})=0.
\end{equation*}

We first present the error bound between $\widehat{\bs{\eta}}_{\gamma}$ with $\gamma \in B(\bs{\gamma}^*, \epsilon_{n2})$ and $\bs{\eta}^*$. We
construct the Newton iterative sequence $\{\bs{\eta}^{(k+1)}\}_{k=0}^\infty$ with  initial value $\bs{\eta}^*$,
where $\bs{\eta}^{(k+1)} = \bs{\eta}^{(k)} - [ F_\gamma'(\bs{\eta}^{(k)})]^{-1} F_\gamma(\bs{\eta}^{(k)})$ and
$F_\gamma'(\bs{\eta})=\partial F_\gamma(\bs{\eta})/\partial \bs{\gamma}^\top$.
If the iterative converges, then the solution lies in the neighborhood of $\bs{\eta}^*$.
This is done by establishing a geometrically fast convergence rate of the iterative sequence.
Details are given in supplementary material.
The error bound is stated below.

\begin{lemma}\label{lemma:alpha:beta-fixed}
If $\bs{\gamma} \in B( \bs{\gamma}^*, \epsilon_{n2} )$ and $b_n = o( (n/\log n)^{1/12})$,
then as $n$ goes to infinity,
with probability at least $1-O(1/n)$, $\widehat{\bs{\eta}}_\gamma$ exists
and satisfies
\[
 \|\widehat{\bs{\eta}}_\gamma - \bs{\eta}^* \|_\infty
 = O_p\left(  b_n^3 \sqrt{\frac{ \log n}{n} } \right)=o_p(1).
\]
Further, if $\widehat{\bs{\eta}}_\gamma$ exists, it is unique.
\end{lemma}

By the compound function derivation law, we have
\begin{eqnarray}\label{equ-derivation-a}
0=\frac{ \partial F_\gamma(\widehat{\bs{\eta}}_\gamma) }{\partial \bs{\gamma}^\top}
 = \frac{ \partial F(\widehat{\bs{\eta}}_\gamma, \gamma) }{\partial \bs{\eta}^\top}
  \frac{\partial \widehat{\bs{\eta}}_\gamma }{  \bs{\gamma}^\top} + \frac{\partial F(\widehat{\bs{\eta}}_\gamma,  \bs{\gamma})}{\partial \bs{ \gamma}^\top},
  \\
  \label{equ-derivation-b}
\frac{ \partial Q_c({ \bs{\gamma}})}{ \partial { \bs{\gamma}}^\top} = \frac{\partial Q(\widehat{\bs{\eta}}_\gamma, { \bs{\gamma}})}{\partial \bs{\eta}^\top}
 \frac{\partial \widehat{\bs{\eta}}_\gamma }{{ \bs{\gamma}}^\top} + \frac{ \partial Q(\widehat{\bs{\eta}}_\gamma, { \bs{\gamma}}) }{ \partial { \bs{\gamma}}^\top}.
\end{eqnarray}
By solving $\partial\widehat{\bs{\eta}}_\gamma / \partial { \bs{\gamma}}^\top$ in \eqref{equ-derivation-a}
and substituting it into \eqref{equ-derivation-b}, we get
the Jacobian matrix
$Q_c^\prime( \bs{\gamma} )$ $(=\partial Q_c(\bs{\gamma})/\partial \bs{\gamma}^\top)$:
\begin{eqnarray}\label{equation:Qc-derivative}
\frac{ \partial Q_c(\bs{\gamma}) }{ \partial \bs{\gamma}^\top }  =
\frac{ \partial Q(\widehat{ \bs{\eta} }_\gamma, \bs{\gamma}) }{ \partial \bs{\gamma}^\top}
 - \frac{ \partial Q(\widehat{ \bs{\eta} }_\gamma, \bs{\gamma}) }{\partial \bs{\eta}^\top}
 \left[\frac{\partial F(\widehat{ \bs{\eta} }_\gamma,\bs{\gamma})}{\partial \bs{\eta}^\top}  \right]^{-1}
\frac{\partial F(\widehat{\bs{\eta}}_\gamma,\bs{\gamma})}{\partial \bs{\gamma}^\top}.
\end{eqnarray}
The asymptotic behavior of $\widehat{\bs{\gamma}}$ crucially depends on the Jacobian matrix
$Q_c^\prime(\bs{\gamma})$.
Since $\widehat{\bs{\eta}}_\gamma$ does not have a closed form, conditions that are directly imposed on $Q_c^\prime(\bs{\gamma})$ are not easily checked.
To derive feasible conditions, we define
\begin{equation}
\label{definition-H}
H(\bs{\eta}, \bs{\gamma}) = \frac{ \partial Q(\bs{\eta}, \bs{\gamma}) }{ \partial \bs{\gamma}} - \frac{ \partial Q(\bs{\eta}, \bs{\gamma}) }{\partial \bs{\eta}} \left[ \frac{\partial F(\bs{\eta}, \bs{\gamma})}{\partial \bs{\eta}} \right]^{-1}
\frac{\partial F(\bs{\eta}, \bs{\gamma})}{\partial \bs{\gamma}},
\end{equation}
which is a general form of $ \partial Q_c(\bs{\gamma}) / \partial \bs{\gamma}$.
Because $H(\bs{\eta}, \bs{\gamma})$ is the Fisher information matrix of the profiled log-likelihood $\ell_c(\bs{\gamma})$,
we assume that it is positively definite. Otherwise, the covariate-$p_0$-model will be ill-conditioned.
When $\bs{\eta}\in B(\bs{\eta}^*, \epsilon_{n1})$, we have the equation:
\begin{equation}\label{equation-H-appro}
\frac{1}{n^2} H(\bs{\eta}, \bs{\gamma}^*) = \frac{1}{n^2} H(\bs{\eta}^*, \bs{\gamma}^*) + o(1),
\end{equation}
whose proof is omitted.
Note that the dimension of $H(\bs{\eta}, \bs{\gamma})$ is fixed and every its entry is a sum of $(n-1)n$  terms.
We assume that there exists a number $\kappa_n$ such that
\begin{equation}\label{condition-Qc-gamma}
\sup_{\bs{\eta}\in B(\bs{\eta}^*, \epsilon_{n1})} \| H^{-1}(\bs{\eta}, \bs{\gamma}^*)\|_\infty \le  \frac{ \kappa_n }{ n^2}.
\end{equation}
If $n^{-2}H(\bs{\eta}, \bs{\gamma}^*)$ converges to a constant matrix, then $\kappa_n$ is bounded.
Because $H(\bs{\eta}, \bs{\gamma}^*)$ is positively definite,
\[
\kappa_n = \sqrt{p}\times \sup_{\bs{\eta}\in B(\bs{\eta}^*, \epsilon_{n1})} 1/\lambda_{\min}(\bs{\eta}),
\]
where  $\lambda_{\min}(\bs{\eta})$ be the smallest eigenvalue of $n^{-2}H(\bs{\eta}, \bs{\gamma}^*)$.
Now we formally state the consistency result.

\begin{theorem}\label{Theorem:con}
If $\kappa_n^2 b_n^{18} = o(n/\log n)$, then
 the MLE $(\widehat{\bs{\eta}}, \widehat{\bs{\gamma}})$ exists with probability at least $1-O(1/n)$, and is consistent in the sense that
\begin{align*}\label{Newton-convergence-rate}
\| \widehat{\bs{\gamma}} - \bs{\gamma}^{*} \|_\infty &=  O_p\left(
 \frac{\kappa_n b_{n}^9 \log n }{ n } \right  )=o_p(1) \\
\| \widehat{ \bs{\eta} } - \bs{\eta}^* \|_\infty &= O_p\left( b_n^3 \sqrt{\frac{\log n}{n}} \right)=o_p(1).
\end{align*}
Further, if $\widehat{\theta}$ exists, it is unique.
\end{theorem}

From the above theorem, we can see that $\widehat{\bs{\gamma}}$ has a convergence rate $1/n$ and
$\widehat{\bs{\eta}}$ has a convergence rate $1/n^{1/2}$, up to a logarithm factor,  if $b_n$ and $\kappa_n$ is bounded above by a constant.
If $\|\bs{\gamma}^{*}\|_\infty$ and $\|\bs{\eta}^*\|_\infty$ are constants, then $b_n$ and $\kappa_n$ are constants as well such that
the condition in Theorem \ref{Theorem:con} holds.

As an application of the $\ell_\infty$-error bound, we use it to recover nonzero signals on node parameters.
We need a separation measure $\Delta_n$ to distinguish between nonzero signals and zero signals for $\beta_i$s.
Define $\Xi =\{i : \beta_i^* > \Delta_n\}$ as the set of signals. Because
\[
\hat{\beta}_i \ge  \beta_i^* - | \hat{\beta}_i - \beta_i^* | = \Delta_n  - O\left( b_n^3 \sqrt{ \frac{\log n }{ n} } \right),~ i\in \Xi,
\]
we immediately have the following corollary.

\begin{corollary}
If $\Delta_n \gg b_n^3 (\log n/n)^{1/2}$ and $\kappa_n^2 b_n^{18} = o(n/\log n)$, then
with probability at least $1-O(n^{-1})$, the set $\Xi$ can be exactly recovered by its MLE $\widehat{\Xi}$,
where $\widehat{\Xi}=\{i:  \hat{\beta}_i\ge  \Delta_n  \}$.
\end{corollary}

\subsection{Asymptotic normality of $\widehat{\bs{\eta}}$ and  $\widehat{\bs{\gamma}}$}

The asymptotic distribution of $\widehat{\bs{\eta}}$ depends crucially on
the inverse of the Fisher information matrix $F^\prime_\gamma(\bs{\eta})$.
It does not have a closed form.
In order to characterize this matrix,
we introduce a general class of matrices that encompass the Fisher matrix.
Given two positive numbers $m$ and $M$ with $M \ge m >0$,
we say the $(2n-1)\times (2n-1)$ matrix $V=(v_{i,j})$ belongs to the class $\mathcal{L}_{n}(m, M)$ if the following holds:
\begin{equation}\label{eq1}
\begin{array}{l}
0\le v_{i,i}-\sum_{j=1}^{2n-1} v_{i,j} \le M, ~~ i=1, \ldots, 2n-1, \\
v_{i,j}=0, ~~ i,j=1,\ldots,n,~ i\neq j, \\
v_{i,j}=0, ~~ i,j=n+1, \ldots, 2n-1,~ i\neq j,\\
m\le v_{i,j}=v_{j,i} \le M, ~~ i=1,\ldots, n,~ j=n+1,\ldots, 2n,~ j\neq n+i, \\
\end{array}
\end{equation}
Clearly, if $V\in \mathcal{L}_{n}(m, M)$, then $V$ is a $(2n-1) \times (2n-1)$ diagonally dominant, symmetric nonnegative
matrix. It must be positively definite.
The definition of $\mathcal{L}_{n}(m, M)$ here is wider than that in   \cite{yan2016aos}, where
the diagonal elements are equal to the row sum excluding themselves for some rows.
 One can easily show that the Fisher information matrix for the vector parameter $ \bs{\eta}$ belongs to $\mathcal{L}_{n}(m, M)$. With some abuse of notation, we use $V$ to denote the Fisher information matrix for the vector
parameter $\bs{\eta}$.

To describe the exact form of elements of $V$, $v_{ij}$ for $i,j=1,\ldots, 2n-1$, $i\neq j$, we define
\begin{equation*}
u_{ij} = \mu^\prime(\pi_{ij}), ~~ u_{ii}=0,~~
u_{i\cdot} = \sum_{j=1}^n u_{ij}, ~~ u_{\cdot j} = \sum_{i=1}^n u_{ij}.
\end{equation*}
Note that $u_{ij}$ is the variance of $a_{ij}$.
Then the elements of $V$ are
\[
v_{ij} = \begin{cases} u_{i\cdot}, & i=j=1, \ldots, n, \\
u_{i,j-n}, &i=1, \ldots, n, j=n+1, \ldots, 2n-1, j\neq i+n, \\
u_{i-n,j}, &i=n+1, \ldots, 2n, j=1, \ldots, n-1, j\neq i-n,\\
u_{\cdot j-n}, & i=j=n+1, \ldots, 2n-1, \\
0, & \mbox{others}.
\end{cases}
\]
\cite{yan2016aos} proposed to approximate the inverse $V^{-1}$ by the matrix $S=(s_{i,j})$, which is defined as
\begin{equation}
\label{definition:S}
s_{i,j}=\left\{\begin{array}{ll}\frac{\delta_{i,j}}{u_{i\cdot}} + \frac{1}{u_{\cdot n}}, & i,j=1,\ldots,n, \\
-\frac{1}{u_{\cdot n}}, & i=1,\ldots, n,~~ j=n+1,\ldots,2n-1, \\
-\frac{1}{u_{\cdot n}}, & i=n+1,\ldots,2n,~~ j=1,\ldots,n-1, \\
\frac{\delta_{i,j}}{u_{\cdot (j-n)}}+\frac{1}{u_{\cdot n}}, & i,j=n+1,\ldots, 2n-1,
\end{array}
\right.
\end{equation}
where $\delta_{i,j}=1$ when $i=j$ and $\delta_{i,j}=0$ when $i\neq j$.

We derive the asymptotic normality of $\widehat{\bs{\bs{\eta}}}$ by approximately representing
$\bs{\widehat{\bs{\eta}}}$ as a function of $\mathbf{d}$ and $\mathbf{b}$ with an explicit expression.
This is done via applying a second Taylor expansion to $F(\widehat{\bs{\eta}}, \widehat{\bs{\gamma}})$ and
showing various remainder terms asymptotically neglect.
Details are given in the supplementary mateiral.

\begin{theorem}\label{Theorem:binary:central}
 If $\kappa_n b_n^{12} = o( n^{1/2}/\log n)$, then for any fixed $k\ge 1$, as $n \to\infty$, the vector consisting of the first $k$ elements of $(\widehat{\bs{\bs{\eta}}}-\bs{\bs{\eta}}^*)$ is asymptotically multivariate normal with mean $\mathbf{0}$ and covariance matrix given by the upper left $k \times k$ block of $S$ defined in \eqref{definition:S}.
\end{theorem}

\begin{remark}
By Theorem \ref{Theorem:binary:central}, for any fixed $i$, as $n\rightarrow \infty$, the asymptotic variance of $\hat{\bs{\eta}}_i$ is $1/v_{i,i}^{1/2}$,
whose magnitude is between $O(n^{-1/2}e^{\|\bs{\bs{\eta}}^*\|_\infty})$ and $O(n^{-1/2})$.
\end{remark}

\begin{remark}
Under the restriction $\alpha_n=\beta_n=0$,
the scaled MLE $\widehat{\nu}$ for the density parameter $\nu$, $ (u_{1\cdot}+u_{\cdot n})^{-1/2}( \widehat{\nu} - \nu )$,
converges in distribution to the standard normality.
\end{remark}

With similar arguments for proving
the asymptotic normality of the restricted MLE for ${\bs{\gamma}}$ in \cite{Yan-Jiang-Fienberg-Leng2018}, we have
the same asymptotic normality of the unrestricted MLE. We only present the result here and the proof is omitted.
Let
\[
V_{\eta\gamma} = \frac{\partial F(\bs{\eta}^*, \bs{\gamma}^*)}{ \partial \bs{\eta}^\top}, ~~
V_{\gamma\gamma} = \frac{\partial Q(\bs{\eta}^*, \bs{\gamma}^* ) }{ \partial \bs{\gamma}^\top }.
\]
Following \cite{Amemiya:1985} (p. 126), the Hessian matrix of the concentrated log-likelihood function $\ell^c (\bs{\gamma}^*)$
is $V_{\gamma\gamma} - V_{\eta\gamma}^\top V^{-1} V_{\eta\gamma}$.
To state the form of the limit distribution of $\hat{\bs{\gamma}}$, define
\begin{equation}\label{eq:I0:beta}
I_n(\bs{\gamma}^*) =  \frac{1}{(n-1)n}(V_{\gamma\gamma} - V_{\eta\gamma}^\top V^{-1} V_{\eta\gamma}).
\end{equation}
Assume that the limit of  $I_n(\bs{\gamma}^*)$ exists as $n$ goes to infinity, denoted by $I_*(\bs{\gamma}^*)$.
Then, we have the following theorem.

\begin{theorem}\label{theorem:covariate:asym}
If $\kappa_n b_n^{12} = o( n^{1/2}/(\log n)^{3/2})$, then as $n$ goes to infinity, the $p$-dimensional vector
$N^{1/2}(\widehat{\bs{\gamma}}-\bs{\gamma}^* )$ is asymptotically multivariate normal distribution
with mean $I_*^{-1} (\bs{\gamma}^*) B_*$ and covariance matrix $I_*^{-1}(\bs{\gamma})$,
where $N=n(n-1)$ and $B_*$ is the bias term:
\begin{equation*}\label{definition:Bstar}
B_*=\lim_{n\to\infty} \frac{1}{2\sqrt{N}} \left[\sum_{i=0}^n \frac{  \sum_{j=0,j\neq i}^n \mu^{\prime\prime}(\pi_{ij}^*) Z_{ij} }
{  \sum_{j=0,j\neq i}^n \mu^\prime(\pi_{ij}^*) }+\sum_{j=0}^n \frac{  \sum_{i=0,i\neq j}^n \mu^{\prime\prime}(\pi_{ij}^*)Z_{ij} }
{ \sum_{i=0,i\neq j}^n \mu^\prime(\pi_{ij}^*) } \right].
\end{equation*}
\end{theorem}

The limiting distribution of $\bs{\widehat{\gamma}}$ is involved with a bias term
\[
\mu_*=\frac{ I_*^{-1}(\bs{\gamma}) B_* }{ \sqrt{n(n-1)}}.
\]
If all parameters $\bs{\gamma}$ and $\bs{\theta}$ are bounded, then
$\mu_*=O( n^{-1/2})$. It follows that $B_*=O(1)$ and $(I_*)_{i,j}=O(1)$ according to their expressions.
Since the MLE $\bs{\widehat{\gamma}}$ is not centered at the true parameter value, the confidence intervals
and the p-values of hypothesis testing constructed from $\bs{\widehat{\gamma}}$ cannot achieve the nominal level without bias-correction under the null: $\bs{\gamma}^*=0$.
This is referred to as the so-called incidental parameter problem in econometric literature [\cite{Neyman:Scott:1948}].
The produced bias is due to the appearance of additional parameters.
As discussed in \cite{Yan-Jiang-Fienberg-Leng2018}, we could use the analytical bias correction formula:
$\bs{\widehat{\gamma}}_{bc} = \bs{\widehat{\gamma}}- \hat{I}^{-1} \hat{B}/\sqrt{n(n-1)}$,
where $\hat{I}$ and $\hat{B}$ are the estimates of $I_*$ and $B_*$ by replacing
$\bs{\gamma}$ and $\bs{\theta}$ in their expressions with their MLEs $\bs{\widehat{\gamma}}$ and
$\bs{\widehat{\theta}}$, respectively.

\section{Numerical Studies}
\label{section:simulation}
In this section, we evaluate the asymptotic results of the MLEs for model \eqref{Eq:density:whole} through simulation studies and a real data example.

\subsection{Simulation studies}

Similar to \cite{Yan-Jiang-Fienberg-Leng2018}, the parameter values take a linear form. Specifically,
we set $\alpha_{i}^* = ic\log n/n$ for $i=0, \ldots, n$ and let $\beta_i^*=\alpha_i^*$, $i=0, \ldots, n$ for simplicity.
Note that there are $n+1$ nodes in the simulations.
We considered four different values for $c$ as $c\in \{0, 0.2, 0.4, 0.6\}$.
By allowing the true values of $\bs{\alpha}$ and $\bs{\beta}$ to grow with $n$, we intended to assess the asymptotic properties under different asymptotic regimes.
Similar to \cite{Yan-Jiang-Fienberg-Leng2018}, we set $p=2$ here.
The first element of the $2$-dimensional node-specific covariate $X_i$ is independently generated from a $Beta(2,2)$ distribution
and the second element follows a discrete distribution taking values $1$ and $-1$ with probabilities $0.3$ and $0.7$.
The covariates formed: $Z_{ij}=(|X_{i1}-X_{j1}|, X_{i2}*X_{j2})^\top$.
For the parameter $\bs{\gamma}^*$, we let it be $(1, 1)^\top$.
The density parameter is $\nu= -\log n/4$.

Note that by Theorems \ref{Theorem:binary:central}, $\hat{\xi}_{i,j} = [\hat{\alpha}_i-\hat{\alpha}_j-(\alpha_i^*-\alpha_j^*)]/(1/\hat{v}_{i,i}+1/\hat{v}_{j,j})^{1/2}$, $\hat{\zeta}_{i,j} = (\hat{\alpha}_i+\hat{\beta}_j-\alpha_i^*-\beta_j^*)/(1/\hat{v}_{i,i}+1/\hat{v}_{n+j,n+j})^{1/2}$, and $\hat{\eta}_{i,j} = [\hat{\beta}_i-\hat{\beta}_j-(\beta_i^*-\beta_j^*)]/(1/\hat{v}_{n+i,n+i}+1/\hat{v}_{n+j,n+j})^{1/2}$
are all asymptotically distributed as standard normal random variables, where $\hat{v}_{i,i}$ is the estimate of $v_{i,i}$
by replacing $(\bs{\gamma}^*, \bs{\eta}^*)$ with $(\bs{\widehat{\gamma}}, \bs{\widehat{\eta}})$.
We record the coverage probability of the 95\% confidence interval, the length of the confidence interval, and the frequency that the MLE does not exist.  The results for $\hat{\xi}_{i,j}$, $\hat{\zeta}_{i,j}$ and $\hat{\eta}_{i,j}$ are similar, thus only the results of $\hat{\xi}_{i,j}$ are reported.
Each simulation is repeated $5,000$ times.

Table \ref{Table:alpha} reports the coverage probability of the 95\% confidence interval for $\alpha_i - \alpha_j$, the length of the confidence interval as well as the frequency that the MLE did not exist.
As we can see, the length of the confidence interval increases as $c$ increases and decreases as $n$ increases, which qualitatively agrees with the theory.  The coverage frequencies are all close to the nominal level when $c=0$ or $0.2$ or $c=0.4$.
When $c=0.6$, the MLE failed to exist with positive frequencies.

{\renewcommand{\arraystretch}{1}
\begin{table}[!h]\centering
\caption{The reported values are the coverage frequency ($\times 100\%$) for $\alpha_i-\alpha_j$ for a pair $(i,j)$ / the length of the confidence interval / the frequency ($\times 100\%$) that the MLE did not exist. The pair $(0,0)$ indicates
those values for the density parameter $\nu$.}
\label{Table:alpha}
\begin{tabular}{ccccccc}
\hline
n       &  $(i,j)$ & $c=0$ & $c=0.2$ & $c=0.4$ & $c=0.6$ \\
\hline
100         &$(1,2)   $&$ 94.9 / 1.28 / 0 $&$ 95.02 / 1.26 / 0 $&$ 94.58 / 1.29 / 0 $&$ 94.76 / 1.35 / 30.88 $ \\
            &$(50,51) $&$ 95.4 / 1.28 / 0 $&$ 94.34 / 1.27 / 0 $&$ 94.98 / 1.39 / 0 $&$ 95.14 / 1.68 / 30.88 $ \\
            &$(99,100)$&$ 94.7 / 1.28 / 0 $&$ 94.54 / 1.31 / 0 $&$ 94.58 / 1.68 / 0 $&$ 97.57 / 2.55 / 30.88 $ \\
            &$(1,100)$ &$ 95.12 / 1.28 / 0 $&$ 94.5 / 1.29 / 0 $&$ 94.8 / 1.5 / 0 $&$ 96.53 / 1.99 / 30.88 $ \\
            &$(1,50)$ &$ 94.64 / 1.28 / 0 $&$ 94.52 / 1.26 / 0 $&$ 94.6 / 1.34 / 0 $&$ 94.39 / 1.52 / 30.88 $ \\
            &$(0,0)$ &$ 94.26 / 1.29 / 0 $&$ 94.66 / 1.27 / 0 $&$ 94.2 / 1.29 / 0 $&$ 94.36 / 1.36 / 30.88 $ \\

&&&&&&\\
200         &$(1,2)     $&$ 95.12 / 0.92 / 0 $&$ 95.3 / 0.89 / 0 $&$ 94.34 / 0.91 / 0 $&$ 94.73 / 0.96 / 7.74 $ \\
            &$(100,101) $&$ 94.86 / 0.92 / 0 $&$ 94.92 / 0.89 / 0 $&$ 94.98 / 1 / 0 $&$ 94.58 / 1.25 / 7.74 $ \\
            &$(199,200) $&$ 94.54 / 0.93 / 0 $&$ 94.74 / 0.92 / 0 $&$ 94.88 / 1.26 / 0 $&$ 96.4 / 2.1 / 7.74 $ \\
            &$(1,200)   $&$ 94.84 / 0.92 / 0 $&$ 94.94 / 0.91 / 0 $&$ 95.52 / 1.11 / 0 $&$ 96.05 / 1.61 / 7.74 $ \\
            &$(1,100) $&$ 95.04 / 0.92 / 0 $&$ 95.66 / 0.89 / 0 $&$ 95.26 / 0.96 / 0 $&$ 94.39 / 1.12 / 7.74 $ \\
            &$(0,0)$ &$ 94.62 / 0.92 / 0 $&$ 94.9 / 0.89 / 0 $&$ 94.78 / 0.91 / 0 $&$ 94.17 / 0.96 / 7.74 $ \\
\hline
\end{tabular}
\end{table}
}

Table \ref{Table:gamma} reports simulation results for the estimate $\bs{\widehat{\gamma}}$
and bias correction estimate  $\bs{\widehat{\gamma}}_{bc} (= \bs{\widehat{\gamma}}- \hat{I}^{-1} \hat{B}/\sqrt{n(n-1)})$ at the nominal level $95\%$.
The coverage frequencies for the uncorrected estimate $\widehat{\gamma}_2$ are visibly below the nominal level
and the bias correction estimates are close to the nominal level when $c\le 0.2$, which
 dramatically improve uncorrected estimates.
As expected, the standard error increases with $c$ and decreases with $n$.

{\renewcommand{\arraystretch}{1}
\begin{table}[!htbp]\centering
\caption{
The reported values are the coverage frequency ($\times 100\%$) of $\gamma_i$ for $i$ with corrected estimates (uncorrected estimates) / length of confidence interval
 /the frequency ($\times 100\%$) that the MLE did not exist ($\bs{\gamma}^*=(1, 1.5)^\top$).
}
\label{Table:gamma}
\begin{tabular}{cclllcc}
\hline
$n$     &   $c$            & $\gamma_1$            & $\gamma_2$ \\
\hline
$100$   & $0$             &$ 93.72 ( 93.72 ) / 0.58 / 0 $&$ 92.72 ( 86.62 ) / 0.11 / 0 $ \\

        & $0.2$           &$ 93.12 ( 92.88 ) / 0.57 / 0 $&$ 93.4 ( 83.58 ) / 0.1 / 0 $ \\

        & $0.4$           &$ 91.26 ( 92.96 ) / 0.62 / 0 $&$ 73.04 ( 86.06 ) / 0.12 / 0 $ \\

        & $0.6$           &$ 85.24 ( 92.62 ) / 0.73 / 30.88 $&$ 43.66 ( 87.44 ) / 0.15 / 30.88 $\\

$200$   & $0$            &$ 92.9 ( 93.94 ) / 0.29 / 0 $&$ 92.38 ( 88.04 ) / 0.06 /  0 $ \\
        & $0.2$      &$ 94.14 ( 93.76 ) / 0.28 / 0 $&$ 93.32 ( 84.22 ) / 0.05 / 0$\\

        & $0.4$            &$ 90.5 ( 93.00 ) / 0.31 / 0 $&$ 72.36 ( 86.76 ) / 0.06 /  0$ \\
        & $0.6$      &$ 83.29 ( 93.84 ) / 0.38 /  7.74 $&$ 40.93 ( 88.64 ) / 0.08 /  7.74 $ \\

\hline
\end{tabular}
\end{table}
}

\subsection{A data example}

\cite{Yan-Jiang-Fienberg-Leng2018} analyzed the friendship network among the $71$ attorneys in Lazega's datasets of lawyers \citep{Lazega:2001} without the density parameter $\nu$,
which can be ownloaded from \url{https://www.stats.ox.ac.uk/~snijders/siena/Lazega_lawyers_data.htm}.
We revisited this data set here and use the covariate-$p_0$-model with the density parameter for analysis.
 The collected covariates at the node level include
 $X_1$=status (1=partner; 2=associate); $X_2$=gender (1=man; 2=woman), $X_3$=location (1=Boston; 2=Hartford; 3=Providence),
$X_4$=years (with the firm), $X_5$=age, $X_6$=practice (1=litigation; 2=corporate) and $X_7$=law school (1: harvard, yale; 2: ucon; 3: other).
We denote the covariates of individual $i$ by $(X_{i1}, \ldots, X_{i7})^T$.
The network graphs with node attributes ``status and practice" were drawn in \cite{Yan-Jiang-Fienberg-Leng2018}.
Here we complement graphs with node attributes ``gender and location" in Figure \ref{figure-data}.
From this figure,  we can see that there are more links between lawyers with the same gender and office location.
It exhibits the homophily of gender and location.

\begin{figure}[t!]
 \centering
  \subfigure{
  \centering
    \includegraphics[width=0.45\linewidth]{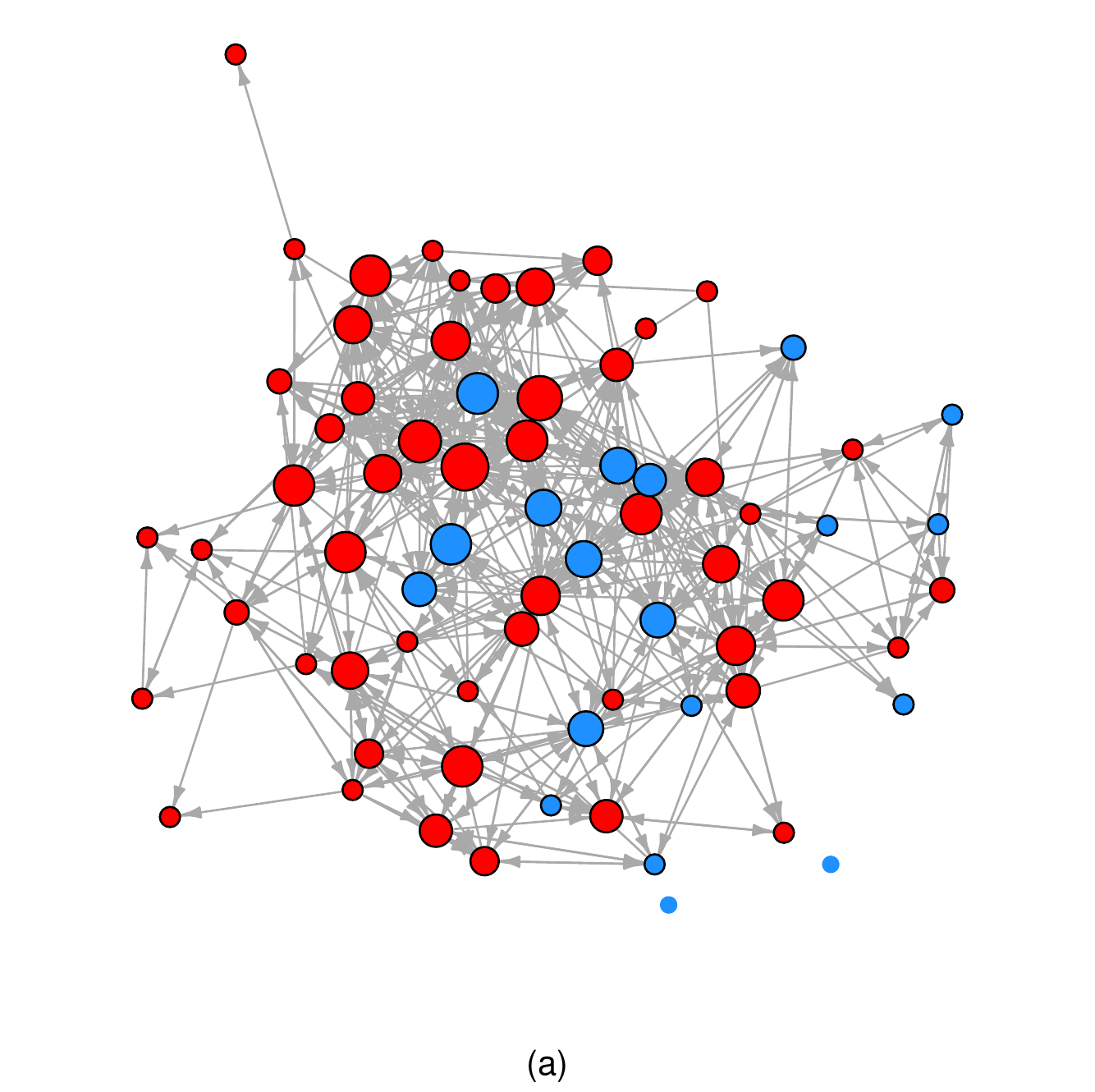}}
  \subfigure{
    \centering
    \includegraphics[width=0.45\linewidth]{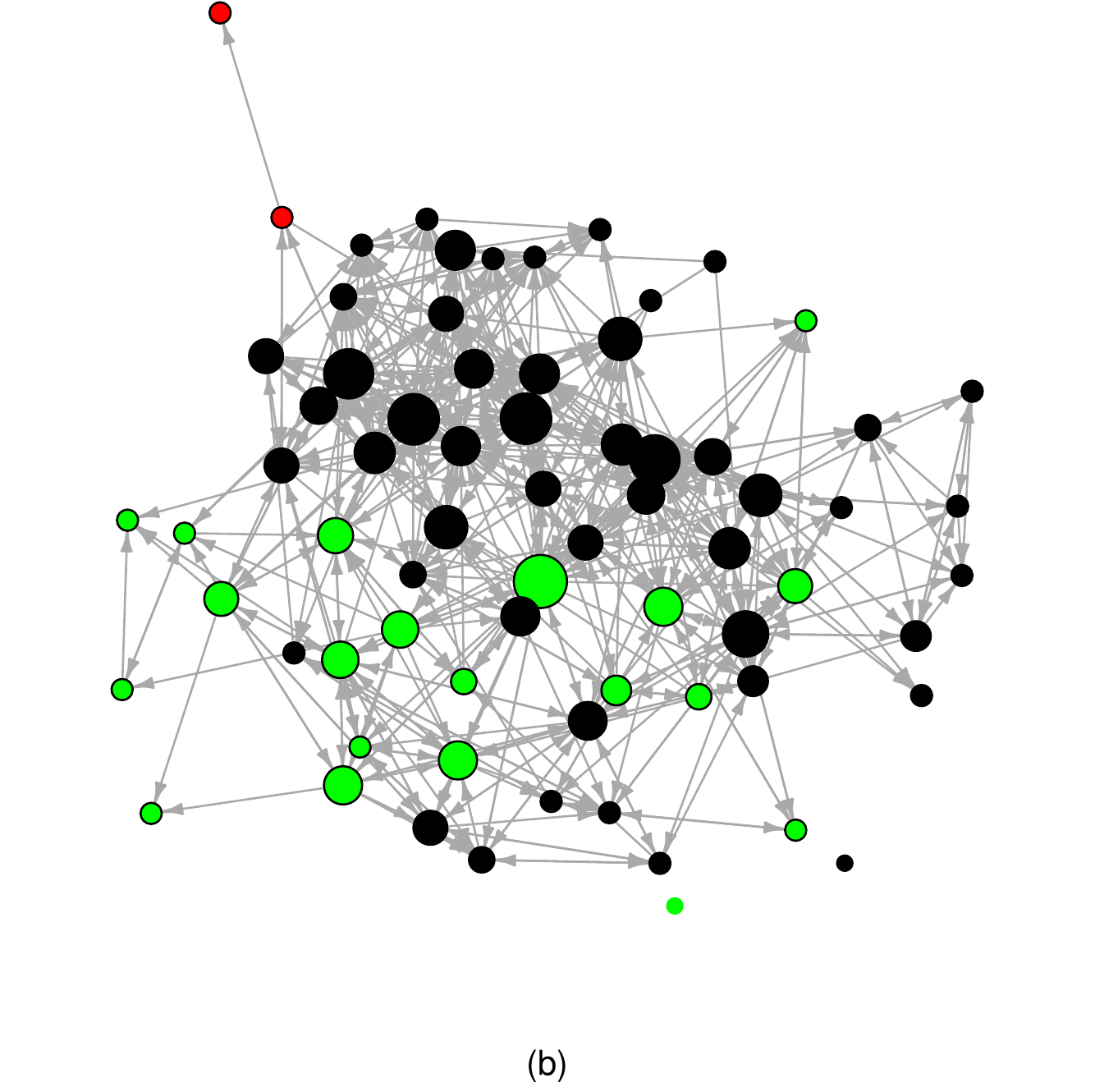}}
 \caption{Visualization of Lazega¡¯s friendship network among lawyers. Two isolated nodes exist.
The vertex sizes are proportional to either nodal in-degrees in (a) or out-degrees in (b).
The positions of the vertices are the same in (a) and (b). For nodes with degrees less than $5$, we set their sizes the same (as a node with degrees $5$). In (a), the colors indicate different
genders (red for male and blue for female), while in (b), the colors represent different locations (black for Boston, green for Hartford, red for Providence).}
\label{figure-data}
\end{figure}

Let $1_{\{\cdot\}}$ be an indicator function.
As in \cite{Yan-Jiang-Fienberg-Leng2018}, we define the similarity between two individuals $i$ and $j$ by the indicator function:
$Z_{ij,k}=1_{\{X_{ik}=X_{jk}\}}$, $k=1, 2, 3, 6, 7$, for categorical variables,  For discrete variables, we define the similarity using absolute distance:
$Z_{ij,k}=|X_{ik}-X_{jk}|$, $k=4,5$.
Individuals are labelled from $1$ to $71$. Before our analysis, we removed those individuals whose in-degrees or out-degrees are zeros, we perform the analysis on the $63$ vertices left.
Since the change of model identification conditions have no influence on the analysis of covariates,
the variables ``status, gender, office location and practice" have significant influence on the link formation
while the influence of ``years, age, and law school" is not significant, as shown in Table 4 in \cite{Yan-Jiang-Fienberg-Leng2018}.

The estimate $\widehat{\nu}$ of the density parameter $\nu$ is $-7.83$ with standard error $1.13$, which leads to a
$p$-value less than $10^{-4}$ under the null $\nu=0$. This indicates a significant sparse network.
The estimators of $\alpha_i$, $\beta_i$ with their estimated standard errors are given in Table \ref{Table:alphabeta:real}, in which $\widehat{\alpha}_{71}=\widehat{\beta}_{71}=0$ as a reference.  The estimates of out- and in-degree parameters  vary widely:
from the minimum $-8.65$ to maximum $-2.21$ for $\widehat{\alpha}_i$s and from $ -1.55$ to $2.79$ for $\widehat{\beta}_i$s.
The minimum, $1/4$ quantile, $3/4$ quantile, and maximum values of $d$ are
$1, 5, 8, 12$, and $25$; those of $b$ are $2, 5, 8, 13$, and $22$.

\begin{table}[!h]\centering
\renewcommand\arraystretch{1.0}
\scriptsize
\vskip5pt
\caption{The estimates of $\alpha_i$ and $\beta_j$ and their standard errors in the Lazega's frienship data set
with $\widehat{\alpha}_{71}=\widehat{\beta}_{71}=0$.}
\vskip5pt
\label{Table:alphabeta:real}
\begin{tabular}{ccc ccc ccc ccc ccc}
\hline
Node & $d_i$  &  $\hat{\alpha}_i$ & $\hat{\sigma}_i$ & $b_j$ &  $\hat{\beta}_i$ & $\hat{\sigma}_j$  && Node & $d_i$  &  $\hat{\alpha}_i$ & $\hat{\sigma}_i$ & $b_j$ &  $\hat{\beta}_i$ & $\hat{\sigma}_j$ \\
\hline
\\
1 &$ 4 $&$ -6.21 $&$ 0.79 $&$ 5 $&$ 0.53 $&$ 0.77 $&& 34 &$ 6 $&$ -5.54 $&$ 0.67 $&$ 11 $&$ 1.18 $&$ 0.61 $ \\
2 &$ 4 $&$ -6.01 $&$ 0.82 $&$ 9 $&$ 1.91 $&$ 0.7 $&& 35 &$ 9 $&$ -4.25 $&$ 0.67 $&$ 10 $&$ 1.55 $&$ 0.68 $ \\
4 &$ 14 $&$ -3.46 $&$ 0.65 $&$ 14 $&$ 2.79 $&$ 0.63 $&& 36 &$ 9 $&$ -5.4 $&$ 0.63 $&$ 11 $&$ 0.77 $&$ 0.61 $ \\
5 &$ 3 $&$ -5.01 $&$ 0.8 $&$ 5 $&$ 1.43 $&$ 0.73 $&& 38 &$ 8 $&$ -5.21 $&$ 0.64 $&$ 13 $&$ 1.42 $&$ 0.6 $ \\
7 &$ 1 $&$ -6.59 $&$ 1.17 $&$ 2 $&$ -0.04 $&$ 0.91 $&& 39 &$ 8 $&$ -5.47 $&$ 0.64 $&$ 13 $&$ 1.14 $&$ 0.61 $ \\
8 &$ 1 $&$ -8.32 $&$ 1.16 $&$ 7 $&$ 0.56 $&$ 0.71 $&& 40 &$ 10 $&$ -5.29 $&$ 0.62 $&$ 8 $&$ 0.21 $&$ 0.64 $ \\
9 &$ 6 $&$ -5.98 $&$ 0.73 $&$ 14 $&$ 2.1 $&$ 0.63 $&& 41 &$ 12 $&$ -5.04 $&$ 0.61 $&$ 17 $&$ 1.42 $&$ 0.59 $ \\
10 &$ 14 $&$ -4.17 $&$ 0.65 $&$ 4 $&$ -0.45 $&$ 0.85 $&& 42 &$ 14 $&$ -4.55 $&$ 0.59 $&$ 9 $&$ 0.54 $&$ 0.63 $ \\
11 &$ 5 $&$ -6.49 $&$ 0.74 $&$ 14 $&$ 1.7 $&$ 0.63 $&& 43 &$ 15 $&$ -4.4 $&$ 0.6 $&$ 13 $&$ 1.21 $&$ 0.6 $ \\
12 &$ 22 $&$ -2.95 $&$ 0.61 $&$ 8 $&$ 0.86 $&$ 0.69 $&& 45 &$ 6 $&$ -5.8 $&$ 0.66 $&$ 4 $&$ -0.63 $&$ 0.74 $ \\
13 &$ 14 $&$ -4.35 $&$ 0.64 $&$ 19 $&$ 2.56 $&$ 0.6 $&& 46 &$ 3 $&$ -5.61 $&$ 0.81 $&$ 5 $&$ 0.53 $&$ 0.74 $ \\
14 &$ 6 $&$ -4.27 $&$ 0.7 $&$ 6 $&$ 1.21 $&$ 0.73 $&& 48 &$ 7 $&$ -5.4 $&$ 0.65 $&$ 4 $&$ -0.39 $&$ 0.74 $ \\
15 &$ 3 $&$ -4.89 $&$ 0.8 $&$ 2 $&$ 0.39 $&$ 0.93 $&& 49 &$ 4 $&$ -6.7 $&$ 0.73 $&$ 6 $&$ -0.42 $&$ 0.68 $ \\
16 &$ 8 $&$ -5.66 $&$ 0.68 $&$ 10 $&$ 0.94 $&$ 0.65 $&& 50 &$ 8 $&$ -4.34 $&$ 0.67 $&$ 8 $&$ 1.15 $&$ 0.68 $ \\
17 &$ 23 $&$ -2.85 $&$ 0.61 $&$ 18 $&$ 2.5 $&$ 0.61 $&& 51 &$ 6 $&$ -4.67 $&$ 0.7 $&$ 7 $&$ 1.11 $&$ 0.7 $ \\
18 &$ 8 $&$ -4.62 $&$ 0.66 $&$ 5 $&$ 0.33 $&$ 0.75 $&& 52 &$ 11 $&$ -5.1 $&$ 0.61 $&$ 14 $&$ 1.12 $&$ 0.6 $ \\
19 &$ 4 $&$ -6.85 $&$ 0.76 $&$ 4 $&$ -0.77 $&$ 0.79 $&& 54 &$ 7 $&$ -5.78 $&$ 0.66 $&$ 11 $&$ 0.68 $&$ 0.62 $ \\
20 &$ 12 $&$ -5.01 $&$ 0.65 $&$ 7 $&$ 0.2 $&$ 0.69 $&& 56 &$ 7 $&$ -5.91 $&$ 0.65 $&$ 10 $&$ 0.39 $&$ 0.63 $ \\
21 &$ 8 $&$ -5.73 $&$ 0.67 $&$ 15 $&$ 1.47 $&$ 0.61 $&& 57 &$ 9 $&$ -5.42 $&$ 0.63 $&$ 12 $&$ 0.87 $&$ 0.61 $ \\
22 &$ 8 $&$ -5.67 $&$ 0.65 $&$ 6 $&$ -0.1 $&$ 0.68 $&& 58 &$ 13 $&$ -3.6 $&$ 0.62 $&$ 12 $&$ 1.83 $&$ 0.64 $ \\
23 &$ 1 $&$ -8.65 $&$ 1.16 $&$ 7 $&$ -0.01 $&$ 0.68 $&& 59 &$ 5 $&$ -5.04 $&$ 0.74 $&$ 4 $&$ 0.12 $&$ 0.8 $ \\
24 &$ 23 $&$ -3.59 $&$ 0.59 $&$ 17 $&$ 1.68 $&$ 0.59 $&& 60 &$ 4 $&$ -6.2 $&$ 0.74 $&$ 8 $&$ 0.47 $&$ 0.65 $ \\
25 &$ 11 $&$ -3.95 $&$ 0.63 $&$ 10 $&$ 1.6 $&$ 0.67 $&& 61 &$ 3 $&$ -6.57 $&$ 0.79 $&$ 3 $&$ -0.88 $&$ 0.8 $ \\
26 &$ 9 $&$ -5.45 $&$ 0.64 $&$ 22 $&$ 2.24 $&$ 0.58 $&& 62 &$ 4 $&$ -6.32 $&$ 0.73 $&$ 5 $&$ -0.38 $&$ 0.71 $ \\
27 &$ 13 $&$ -4.54 $&$ 0.61 $&$ 17 $&$ 2.02 $&$ 0.59 $&& 64 &$ 19 $&$ -3.71 $&$ 0.58 $&$ 14 $&$ 1.55 $&$ 0.6 $ \\
28 &$ 11 $&$ -3.91 $&$ 0.64 $&$ 9 $&$ 1.32 $&$ 0.69 $&& 65 &$ 22 $&$ -3.68 $&$ 0.58 $&$ 8 $&$ 0.32 $&$ 0.64 $ \\
29 &$ 10 $&$ -4.81 $&$ 0.62 $&$ 10 $&$ 1.09 $&$ 0.62 $&& 66 &$ 15 $&$ -4.56 $&$ 0.59 $&$ 3 $&$ -0.97 $&$ 0.8 $ \\
30 &$ 6 $&$ -5.26 $&$ 0.71 $&$ 5 $&$ -0.1 $&$ 0.78 $&& 67 &$ 4 $&$ -6.5 $&$ 0.73 $&$ 3 $&$ -1.04 $&$ 0.79 $ \\
31 &$ 25 $&$ -2.21 $&$ 0.59 $&$ 14 $&$ 2.21 $&$ 0.64 $&& 68 &$ 6 $&$ -5.81 $&$ 0.68 $&$ 5 $&$ -0.32 $&$ 0.72 $ \\
32 &$ 4 $&$ -5.86 $&$ 0.79 $&$ 7 $&$ 0.54 $&$ 0.74 $&& 69 &$ 5 $&$ -6.13 $&$ 0.7 $&$ 4 $&$ -0.64 $&$ 0.74 $ \\
33 &$ 12 $&$ -4.03 $&$ 0.64 $&$ 2 $&$ -1.55 $&$ 1.01 $&& 70 &$ 7 $&$ -5.5 $&$ 0.65 $&$ 5 $&$ -0.25 $&$ 0.71 $ \\
34 &$ 6 $&$ -5.54 $&$ 0.67 $&$ 11 $&$ 1.18 $&$ 0.61$\\
\hline
\end{tabular}

\end{table}

\section{Discussion}
\label{section:summary}

In this paper, we have derived the $\ell_\infty$-error between the MLE and its true values and
established the asymptotic normality of the MLE in the covariate-$p_0$-model when the number of vertices goes to infinity.
 Note that the conditions imposed on $b_n$ and $\kappa_n$
in Theorems \ref{Theorem:con}--\ref{theorem:covariate:asym} may not be the best possible.
In particular, the conditions guaranteeing the asymptotic normality seem stronger than those guaranteeing the consistency.
It would be  interesting to investigate whether these bounds can be improved.

As discussed in \cite{Yan-Jiang-Fienberg-Leng2018},
there is an implicit taste for the reciprocity parameter in the $p_1$-model [\cite{Holland:Leinhardt:1981}], although we do not
include this parameter.
If similarity terms  are included, then there is a tendency toward reciprocity among nodes sharing similar node features. That would alleviate the lack of a reciprocity term to some extent.
To measure the reciprocity of dyads, it is natural to
incorporate the model term $\rho\sum_{i<j}a_{ij}a_{ji}$ as in the $p_1$ model
into the covariate-$p_0$-model.
In \cite{Yan:Leng:2013}, empirical results show that there are central limit theorems for the MLE in the $p_1$ model
without covariates. Nevertheless,
although only one new parameter is added, the problem of investigating the asymptotic
theory of the MLEs becomes more challenging. In particular, the Fisher information matrix for the parameter vector $(\rho, \alpha_1,\ldots,\alpha_n, \beta_1, \ldots, \beta_{n-1})$ is not diagonally dominant and thus does not belong to the class $\mathcal{L}_{n}(m, M)$.
In order to generalize the method here, a new approximate matrix with high accuracy of the inverse of the Fisher information matrix is needed.
It is beyond of the present paper to investigate their asymptotic theory.

\renewcommand{\baselinestretch}{1.2}\selectfont

\section{Appendix: Proofs for theorems}
\label{section:proofs}
We only give the proof of Theorem \ref{Theorem:con} here.
The proof of Theorem \ref{Theorem:binary:central} is put in the supplementary material.

Let $F(x): \R^n \to \R^n$ be a function vector on $x\in\R^n$. We say that a Jacobian matrix $F^\prime(x)$ with $x\in \R^n$ is Lipschitz continuous on a convex set $D\subset\R^n$ if
for any $x,y\in D$, there exists a constant $\lambda>0$ such that
for any vector $v\in \R^n$ the inequality
\begin{equation*}
\| [F^\prime (x)] v - [F^\prime (y)] v \|_\infty \le \lambda \| x - y \|_\infty \|v\|_\infty
\end{equation*}
holds.
We will use the Newton iterative sequence to establish the existence and consistency of the moment estimator.
\cite{Gragg:Tapia:1974} gave the optimal error bound for the Newton method under the Kantovorich conditions
[\cite{Kantorovich1948Functional}].

\begin{lemma}[\cite{Gragg:Tapia:1974}]\label{pro:Newton:Kantovorich}
Let $D$ be an open convex set of $\R^n$ and $F:D \to \R^n$ a differential function
with a Jacobian $F^\prime(x)$ that is Lipschitz continuous on $D$ with Lipschitz coefficient $\lambda$.
Assume that $x_0 \in D$ is such that $[ F^\prime (x_0) ]^{-1} $ exists,
\begin{eqnarray*}
\| [ F^\prime (x_0 ) ]^{-1} \|_\infty  \le \aleph,~~ \| [ F^\prime (x_0) ]^{-1} F(x_0) \|_\infty \le \delta, ~~ \rho= 2 \aleph \lambda \delta \le 1,
\\
B(x_0, t^*) \subset D, ~~ t^* = \frac{2}{\rho} ( 1 - \sqrt{1-\rho} ) \delta = \frac{ 2\delta }{ 1 + \sqrt{1-\rho} }\le 2\delta.
\end{eqnarray*}
Then: (1) The Newton iterations $x_{k+1} = x_k - [ F^\prime (x_k) ]^{-1} F(x_k)$ exist and $x_k \in B(x_0, t^*) \subset D$ for $k \ge 0$. (2)
$x^* = \lim x_k$ exists, $x^* \in \overline{ B(x_0, t^*) } \subset D$ and $F(x^*)=0$.
\end{lemma}

\subsection{Proof of Theorem \ref{Theorem:con}}
\label{subsection:Theorem:con}

To show Theorem \ref{Theorem:con}, we need three lemmas below.

\begin{lemma}\label{lemma-Q-Lip}
Let $D=B(\bs{\gamma}^*, \epsilon_{n2}) (\subset \R^{p})$.
If $\| F(\bs{\eta}^*, \bs{\gamma}^*) \|_\infty = O( (n\log n)^{1/2} )$, then
$ Q_c(\bs{\gamma})$ is Lipschitz continuous on $D$ with the Lipschitz coefficient  $O(b_n^9n^2)$.
\end{lemma}

\begin{lemma}\label{lemma-diff-F-Q}
With probability at least $1-O(1/n)$, we have
\begin{equation}
\| F(\bs{\eta}^*, \bs{\gamma}^*) \|_\infty \le (n\log n)^{1/2} , ~~\| Q(\bs{\eta}^*, \bs{\gamma}^*) \|_\infty \le  z_{\max}n (\log n)^{1/2},
\end{equation}
where $z_{\max}:=\max_{i,j} \|Z_{ij}\|_\infty$.
\end{lemma}

\begin{lemma}
\label{lemma-order-Q-beta}
The difference between $Q(\widehat{\bs{\eta}}_{\gamma}^*, \bs{\gamma}^*)$ and $Q(\bs{\eta}^*, \bs{\gamma}^*)$ is
\[
 \|Q(\widehat{\bs{\eta}}_{\bs{\gamma}}^*, \bs{\gamma}^*)-Q(\bs{\eta}^*, \bs{\gamma}^*)\|_\infty  = O_p( b_n^9 n\log n).
\]
\end{lemma}

Now we are ready to prove  Theorem \ref{Theorem:con}.

\begin{proof}[Proof of Theorem \ref{Theorem:con}]

We construct the Newton iterative sequence to show the consistency. It is sufficient to verify the
Newton-Kantovorich conditions in Lemma \ref{pro:Newton:Kantovorich}.
We set $\bs{\gamma}^*$ as the initial point $\bs{\gamma}^{(0)}$ and $\bs{\gamma}^{(k+1)}=\bs{\gamma}^{(k)} - [Q_c^\prime(\bs{\gamma}^{(k)})]^{-1}Q_c(\bs{\gamma}^{(k)})$.

By Lemma \ref{lemma:alpha:beta-fixed}, $\widehat{\bs{\eta}}_{\gamma^*}$ exists with probability approaching one and satisfies
we have
\[
\| \widehat{\bs{\eta}}_{\gamma^*} - \bs{\eta}^* \|_\infty = O_p\left( b_n^3\sqrt{\frac{\log n}{n}} \right). 
\]
Therefore, $Q_c(\bs{\gamma}^{(0)})$ and $Q_c^\prime(\bs{\gamma}^{(0)})$ are well defined.

Recall the definition of $Q_c(\bs{\gamma})$ and $Q(\bs{\eta}, \bs{\gamma})$ in \eqref{definition-Q} and \eqref{definition-Qc}.
By Lemmas \ref{lemma-diff-F-Q} and \ref{lemma-order-Q-beta}, we have
\begin{eqnarray*}
\|Q_c(\bs{\gamma}^*)\|_\infty  & \le &  \|Q(\bs{\eta}^*, \bs{\gamma}^*)\|_\infty + \|Q(\widehat{\bs{\eta}}_{\bs{\gamma}^*}, \bs{\gamma}^*)-Q(\bs{\eta}^*, \bs{\gamma}^*)\|_\infty\\
&=& O_p\left(  b_n^9 n \log n \right).
\end{eqnarray*}
By Lemma \ref{lemma-Q-Lip}, $\lambda=n^2  b_{n}^{9} $.
By \eqref{condition-Qc-gamma}, we have
\[
\aleph=\| [Q_c^\prime(\bs{\gamma}^*)]^{-1} \|_\infty = O ( \kappa_n n^{-2}).
\]
Thus,
\[
\delta = \| [Q_c^\prime(\bs{\gamma}^*)]^{-1} Q_c(\bs{\gamma}^*) \|_\infty =  O_p\left(
 \frac{\kappa_n b_{n}^9 \log n }{ n } \right  ).
\]
 As a result, if $\kappa_n^2 b_{n}^{18}=o(n/\log n)$, then
\[
\rho=2\aleph \lambda \delta =
O( \frac{ \kappa_n^2 b_{n}^{18}  \log n }{n})=o(1).
\]
By Lemma \ref{pro:Newton:Kantovorich}, with probability $1-O(n^{-1})$, the limiting point of the sequence $\{\bs{\gamma}^{(k)}\}_{k=1}^\infty$ exists denoted by $\widehat{\bs{\gamma}}$ and satisfies
\[
\| \widehat{\bs{\gamma}} - \bs{\gamma}^* \|_\infty = O(\delta).
\]
By Lemma \ref{lemma:alpha:beta-fixed}, $\widehat{\bs{\eta}}_{\widehat{\bs{\gamma}}}$ exists and
$( \widehat{\bs{\eta}}_{\widehat{\bs{\gamma}}},\widehat{\bs{\gamma}})$ is the MLE.
It completes the proof.
\end{proof}

\setlength{\itemsep}{-1.5pt}
\setlength{\bibsep}{0ex}
\bibliography{Reference}
\bibliographystyle{apa}

\end{document}